\documentclass{amsart}
\usepackage{amssymb,amstext,amsmath,amscd,amsthm,amsfonts,enumerate,graphicx,latexsym,stmaryrd}
\usepackage[usenames]{color}
\usepackage[all]{xy}

\newtheorem{thm}{Theorem}[section]
\newtheorem{lem}[thm]{Lemma}
\newtheorem{prop}[thm]{Proposition}
\newtheorem{cor}[thm]{Corollary}
\theoremstyle{definition}
\newtheorem{dfn}[thm]{Definition}

\newtheorem{rem}[thm]{Remark}
\newtheorem{conv}[thm]{Convention}

\theoremstyle{remark}
\newtheorem*{ac}{Acknowlegments}

\numberwithin{equation}{thm}

\def\ass{\operatorname{Ass}}
\def\supp{\operatorname{Supp}}
\def\ker{\operatorname{Ker}}

\def\depth{\operatorname{depth}}

\def\height{\operatorname{ht}}

\def\le{\leqslant}
\def\m{\mathfrak{m}}

\def\p{\mathfrak{p}}
\def\q{\mathfrak{q}}

\def\m{\mathfrak{m}}
\def\a{\mathfrak{a}}
\def\b{\mathfrak{b}}
\def\c{\mathfrak{c}}

\def\s{\mathfrak{s}}
\def\g{\mathfrak{g}}

\def\spec{\operatorname{Spec}}
\def\syz{\Omega}

\def\U{\mathrm{U}}
\def\C{\mathrm{C}}

\def\V{\mathrm{V}}

\def\Rfd{\operatorname{Rfd}}

\def\cmd{\operatorname{cmd}}
\def\grade{\operatorname{grade}}
\def\dim{\operatorname{dim}}

\begin{document}
\allowdisplaybreaks
\title{Faltings' annihilator theorem and almost Cohen-Macaulay rings}
\author{Glenn ando}
\address{Graduate School of Mathematics, Nagoya University, Furocho, Chikusaku, Nagoya 464-8602, Japan}
\email{m21003v@math.nagoya-u.ac.jp}
\begin{abstract}
Faltings' annihilator theorem is an important result in local cohomology theory.
Recently, Doustimehr and Naghipour generalized the Falitings' annihilator theorem.
They proved that if $R$ is a homomorphic image of a Gorenstein ring, then $f_\a^\b(M)_n = \lambda_\a^\b(M)_n$, where $f_\a^\b(M)_n := \inf\{i \in \mathbb{N} \mid \dim{\supp(\b^t H_\a^i(M))} \geq n \text{ for all } t\in \mathbb{N}\}$ and $\lambda_\a^\b(M)_n := \inf\{\lambda_{\a R_\p}^{\b R_\p}(M_\p) \mid \p\in\spec{R} \text{ with } \dim{R/\p} \geq n\}$.
In this paper, we study the relation between $f_\a^\b(M)_n$ and $\lambda_\a^\b(M)_n$, and prove that if $R$ is an almost Cohen-Macaulay ring, then $f_\a^\b(M)_n \geq \lambda_\a^\b(M)_n - \cmd{R}$.
Using this result, we prove that if $R$ is a homomorphic image of a Cohen-Macaulay ring, then $f_\a^\b(M)_n = \lambda_\a^\b(M)_n$.
\end{abstract}
\maketitle
\section{Introduction}
Throughout this paper, let $R$ be a commutative noetherian ring.

The annihilators of local cohomology modules have been widely studied in local cohomology theory.
For example, Faltings' annihilator theorem \cite{Fal} states that if $R$ is a homomorphic image of a regular ring and $M$ is a finitely generated $R$-module, then there exists an integer $n$ such that $\b^n H_{\a}^i(M) = 0$ for all ideals $\a$ and $\b$ of $R$ and for all $i < \inf\{\depth M_\p + \height(\a + \p)/\p \mid \p \in \spec{R}\setminus\V(\b) \}$; see also \cite[Theorem 9.4.16]{BS}.
In other words, it states that $f_\a^\b(M) = \lambda_\a^\b(M)$, where 
\begin{align*}
    f_\a^\b(M) &:= \inf\{i \in \mathbb{N} \mid \b^t H_\a^i(M) \neq 0 \text{ for all } t \in \mathbb{N}\},\\
    \lambda_\a^\b(M) &:= \inf\{\depth M_\p + \height(\a + \p)/\p \mid \p \in \spec{R}\setminus\V(\b) \}.
\end{align*}
In 2004, Khashyarmanesh and Salarian \cite{sar} proved that Faltings' annihilator theorem holds over a homomorphic image of a Gorenstein ring.
In 2008, Kawasaki proved a general theorem \cite[Theorem 1.1]{kawa}, which implies that Faltings' annihilator theorem holds over a homomorphic image of a Cohen-Macaulay ring.
Recently, Doustimehr and Naghipour \cite{Dou} generalized Falitings' annihilator theorem.
They defined the two invariants called the $n$-th $\b$-finiteness dimension of $M$ relative to $\a$ and the $n$-th $\b$-minimum $\a$-adjusted depth of $M$ by 
\begin{align*}
    f_\a^\b(M)_n &:= \inf\{i \in \mathbb{N} \mid \dim{\supp(\b^t H_\a^i(M))} \geq n \text{ for all } t\in \mathbb{N}\},\\
    \lambda_\a^\b(M)_n &:= \inf\{\lambda_{\a R_\p}^{\b R_\p}(M_\p) \mid \p\in\spec{R} \text{ with } \dim{R/\p} \geq n \}.
\end{align*}
Note that $f_\a^\b(M)_n$ and $\lambda_\a^\b(M)_n$ are both positive integers or $\infty$, that $f_\a^\b(M)_0 = f_\a^\b(M)$ and $\lambda_\a^\b(M)_0 = \lambda_\a^\b(M)$, and that $f_\a^\b(M)_n \leq \lambda_\a^\b(M)_n$.
They proved that if $R$ is a homomorphic image of a Gorenstein ring and $M$ is a finitely generated $R$-module, then $f_\a^\b(M)_n = \lambda_\a^\b(M)_n$ for all ideals $\a$ and $\b$ of $R$ with $\b \subseteq \a$.
In this paper, we investigate the relationship between $f_\a^\b(M)_n$ and $\lambda_\a^\b(M)_n$ over an almost Cohen-Macaulay ring by using the results about large restricted flat dimension given in \cite{GA}.
The main result is the following theorem; recall that the Cohen-Macaulay defect of $R$ is defined by
$$
\cmd{R} := \sup\{\height{\p}-\depth{R_{\p}}\mid\p\in\spec{R}\},
$$
and that $R$ is called almost Cohen-Macaulay if $\cmd{R} \le 1$.
\begin{thm}\label{11}
Assume that $R$ is an almost Cohen-Macaulay ring.
Let $\a$ and $\b$ be ideals of $R$ such that $\b \subseteq \a$, and let $M$ be a finitely generated $R$-module.
Then, for each $n\in\mathbb{N}$, 
$$
f_\a^\b(M)_n \geq \lambda_\a^\b(M)_n - \cmd{R}.
$$
\end{thm}
The following is an immediate consequence of Theorem \ref{11}.
This is a generalization of \cite[Theorem 2.10]{sar}.
\begin{cor}\label{555}
Assume that $R$ is an almost Cohen-Macaulay ring.
Let $\a$ and $\b$ be ideals of $R$ such that $\b \subseteq \a$, and let $M$ be a finitely generated $R$-module.
Then, for each $n\in\mathbb{N}$, 
$$
f_\a^\b(M) \geq \lambda_\a^\b(M) - \cmd{R}.
$$
In particular, there exists $n\in\mathbb{N}$ such that $\b^n H_{\a}^i(M) = 0$ for all $i < \lambda_\a^\b(M) - \cmd{R}$.
\end{cor}
Theorem \ref{11} yields to the following corollary.
The latter assertion refines Doustimehr and Naghipour's result.
\begin{cor}\label{12}
Let $\a$ and $\b$ be ideals of $R$ such that $\b \subseteq \a$, and let $M$ be a finitely generated $R$-module.
\begin{enumerate}[\rm(1)]
    \item Assume that $R$ is a homomorphic image of an almost Cohen-Macaulay ring.
    Then, for each $n\in\mathbb{N}$,
    $$
    f_\a^\b(M)_n \geq \lambda_\a^\b(M)_n - 1.
    $$
    \item Assume that $R$ is a homomorphic image of a Cohen-Macaulay ring.
    Then, for each $n\in\mathbb{N}$,
     $$f_\a^\b(M)_n = \lambda_\a^\b(M)_n.
    $$
\end{enumerate}
\end{cor}
Moreover, we give an elementary proof of the following assertion rather than using the result in \cite{kawa}.
\begin{cor}\label{777}
Let $\a$ and $\b$ be ideals of $R$ such that $\b \subseteq \a$, and let $M$ be a finitely generated $R$-module.
Assume that $R$ is a homomorphic image of a Cohen-Macaulay ring.
    Then,
     $$
     f_\a^\b(M) = \lambda_\a^\b(M).
     $$
\end{cor}
Another important result in local cohomology theory is the local-global principle for finiteness dimension, which states that $f_\a(M) = \inf\{f_{\a R_\p}(M_\p) \mid \p\in\spec{R} \}$; see \cite[Theorem 9.6.2]{BS}.
In \cite{ANEW}, Asadollahi and Naghipour defined invariant called the upper $n$-th $\b$-minimum $\a$-adjusted depth of $M$ by 
$$
f_\a^\b(M)^n := \inf\{f_{\a R_\p}^{\b R_\p}(M_\p) \mid \p\in\spec{R} \text{ with } \dim{R/\p} \geq n \}.
$$
By \cite[Proposition 3.4]{DNEW}, $f_\a(M)_n$ and $f_\a(M)^n$ are equal, where $f_\a(M)_n := f_\a^\a(M)_n$ and $f_\a(M)^n:=f_\a^\a(M)^n$.
This is a generalization of the local-global principle for finiteness dimension.
It is natural to ask when $f_\a^\b(M)_n$ equals $f_\a^\b(M)^n$.
In \cite{BRS}, Brodmann, Rotthaus, and Sharp proved that if $R$ is a ring such that $\dim{R} \leq 4$, then $f_\a^\b(M) = \inf\{f_{\a R_\p}^{\b R_\p}(M_\p) \mid \p\in\spec{R} \}$.
We generalize this result and show the following.
\begin{thm}\label{13}
Let $R$ be a ring such that $\dim{R} \leq 4$, and let $M$ be a finitely generated $R$-module.
Then, for each $n\in\mathbb{N}$,
$$
f_\a^\b(M)_n = f_\a^\b(M)^n.
$$
\end{thm}
Theorem \ref{13} gives rise to the following corollary, which states that if $R$ is a ring such that $\dim{R} \leq 2$, then the assertion of Corollary \ref{12} (1) always holds.
\begin{cor}\label{140}
Let $R$ be a ring such that $\dim{R} \leq 2$, and let $M$ be a finitely generated $R$-module.
Then, for each $n\in\mathbb{N}$,
$$f_\a^\b(M)_n \geq \lambda_\a^\b(M)_n - 1.
$$
\end{cor}
The organization of this paper is as follows.
In Section 2, we state our conventions, basic notions, and their properties for later use.
In Section 3, we investigate the relationship between $f_\a^\b(M)_n$ and $\lambda_\a^\b(M)_n$ over an almost Cohen-Macaulay ring, and prove Theorem \ref{11} and Corollaries \ref{555}, \ref{12} and \ref{777}.
In Section 4, we generalize the results in \cite{BRS}, and prove Theorem \ref{13} and Corollary \ref{140}.
\section{Basic definition and properties}
In this section, we give several definitions and their properties.
We begin with our convention.

\begin{conv}
All rings are commutative noetherian rings with identity. Let $R$ be a (commutative noetherian) ring, and let $M$ be an $R$-module. 
The symbol $\mathbb{N}$ denotes the set of non-negative integers, let $n\in\mathbb{N}$. 
\end{conv}

First, we give some notations.
\begin{dfn}
\begin{enumerate}[\rm(1)]
\item For any ideal $\a$ of $R$, we denote by $\V(\a)$ (resp. $\U(\a)$) the set of prime ideals of $R$ containing (resp. contained in) $\a$.
\item If $T$ is an arbitrary subset of $\spec{R}$, then we set
$$
T_{\geq n} := \{\p\in T \mid \dim{R/\p} \geq n\}.
$$
\item We denote by $\syz_R^n M$ the $n$-th syzygy module of a finitely generated $R$-module $M$.
Note that $\syz_R^n M$ is uniquely determined up to projective summands.
\end{enumerate}
\end{dfn}
Next, we recall the definitions of Cohen-Macaulay defect and almost Cohen-Macaulay rings.
\begin{dfn}
\begin{enumerate}[\rm(1)]
    \item The Cohen-Macaulay defect of $R$ is defined by 
    $$
    \cmd{R} := \sup\{\height{\p}-\depth{R_{\p}}\mid\p\in\spec{R}\}.
    $$
    \item A ring $R$ is called an almost Cohen-Macaulay ring if $\cmd{R}\leq1$.
\end{enumerate}
\end{dfn}
We recall the definition of large restricted flat dimension and some results in \cite{GA}.
\begin{dfn}
Let $M$ be a finitely generated $R$-module.
The {\em large restricted flat dimension} of $M$ is defined by 
$$
\Rfd_R M = \sup_{\p \in \spec{R}} \left\{\depth{R_\p} - \depth{M_\p} \right\}.
$$
Note that $\Rfd_{R_\p}{M_\p} = \sup_{\q \in \U(\p)} \{ \depth{R_\q} - \depth{M_\q}\}$ holds for each $\p \in \spec{R}$.
We should remark that the original definition of large restricted flat dimension is different; it is similar to the definition of flat dimension from which its name comes, and the equality in the above definition turns out to hold.
One has $\Rfd_R M \in \mathbb{N} \cup \{-\infty \}$, and $\Rfd_R M = -\infty$ if and only if $M = 0$.
For the details of large restricted flat dimension, we refer the reader to \cite[Theorem 1.1]{AIL} and \cite[Theorem 2.4, 2.11 and Observation 2.10]{CFF}.
\end{dfn}
\begin{prop}\label{225}
\rm{(\cite[Theorem 4.4(8)]{GA})}
Let $M$ be a finitely generated $R$-module.
Assume that $R$ is an almost Cohen-Macaulay ring.
Then, for each $\p\in\spec{R}$, the following are equivalent.
\begin{enumerate}[\rm(1)]
    \item There is an inequality $\Rfd_{R_\p}M_\p \leq 0$.
    \item There exists $s\in R\setminus\p$ such that $sH_I^i(M)=0$ for all ideal $I$ of $R$ and for all $i < \grade(I,R)$.
\end{enumerate}
\end{prop}
\begin{prop}\label{226}
\rm{(\cite[Corollary 4.8]{GA})}
Let $M$ be a finitely generated $R$-module, and let $\p$ be a prime ideal of $R$.
Then there exists an element $s \in R \setminus \p$ such that 
$s H_I^i(M) = 0$ for all ideals $I$ of $R$ and for all $i < \grade(I,R) - \Rfd_{R_\p} M_\p$.
\end{prop}
Next, we recall the definition of modules in dimension $<n$, which is introduced by Asadollahi and Naghipour \cite{AN}.
\begin{dfn}
An $R$-module $M$ is said to be in dimension $<n$, if there is a finitely generated submodule $N$ of $M$ such that $\dim{\supp(M/N)} < n$.
Note that if $M$ is in dimension $<n$, then $M_{\p}$ is finitely generated for all $\p\in\spec{R}_{\geq n}$.
\end{dfn}

Finally, we recall some invariants, which are needed to state our results.
The invariants (1) and (2) were introduced by Brodmann and Sharp in \cite{BS}, (3) and (4) by Doustimehr and Naghipour in \cite{Dou}, and (5) by Asadollahi and Naghipour in \cite{ANEW}.
\begin{dfn}
Let $\a$ and $\b$ be ideals of $R$ such that $\b \subseteq \a$, and let $M$ be in dimension $< n$.
\begin{enumerate}[\rm(1)]
    \item The $\b$-finiteness dimension of $M$ relative to $\a$ is defined by 
    $$
    f_\a^\b(M) := \inf\{i \in \mathbb{N} \mid \b^t H_\a^i(M) \neq 0 \text{ for all } t \in \mathbb{N}\}.
    $$
    \item The $\b$-minimum $\a$-adjusted depth of $M$ is defined by 
    $$
    \lambda_\a^\b(M) := \inf\{\depth M_\p + \height(\a + \p)/\p \mid \p \in \spec{R}\setminus\V(\b) \}.
    $$
    \item The $n$-th $\b$-finiteness dimension of $M$ relative to $\a$ is defined by 
    $$
    f_\a^\b(M)_n := \inf\{i \in \mathbb{N} \mid \dim{\supp(\b^t H_\a^i(M))} \geq n \text{ for all } t \in \mathbb{N}\}.
    $$
    \item The $n$-th $\b$-minimum $\a$-adjusted depth of $M$ is defined by 
    $$
    \lambda_\a^\b(M)_n := \inf\{\lambda_{\a R_\p}^{\b R_\p}(M_\p) \mid \p\in\spec{R}_{\geq n} \}.
    $$
    \item The upper $n$-th $\b$-minimum $\a$-adjusted depth of $M$ is defined by 
    $$
    f_\a^\b(M)^n := \inf\{f_{\a R_\p}^{\b R_\p}(M_\p) \mid \p\in\spec{R}_{\geq n} \}.
    $$
\end{enumerate}
Note that all the invariants defined above are positive integers or $\infty$ and that $f_\a^\b(M)_0 = f_\a^\b(M)$ and $\lambda_\a^\b(M)_0 = \lambda_\a^\b(M)$.
\end{dfn}

\begin{prop}\label{100}
\rm(\cite[Lemma 2.1]{ANEW}).
Let $\a$ and $\b$ be ideals such that $\b \subseteq \a$, and let $M$ be a finitely generated $R$-module.
Then,
$$
f_\a^\b(M)_n \leq f_\a^\b(M)^n \leq \lambda_\a^\b(M)_n.
$$
\end{prop}
\begin{rem}\label{14}
Let $\a$ and $\b$ be ideals of $R$ such that $\b \subseteq \a$, and let $M$ be in dimension $< n$.
\begin{enumerate}[\rm(1)]
    \item By \cite[Theorem 4.2.1 and Proposition 8.1.2]{BS}, there exists the exact sequence 
    $$
    \cdots\to(H_{\a}^{i-1}(M))_s \to H_{\a + Rs}^i(M) \to H_{\a}^i(M)\to\cdots.
    $$
    From this, we see that $f_{\a}^{\b}(M)_n \leq f_{\a + Rs}^{\b}(M)_n$. Hence, by induction, we obtain $f_{\a}^{\b}(M)_n \leq f_{\a'}^{\b}(M)_n$ for all $\a'\in\V(\a)$.
    Also, it is easy to see that $\lambda_{\a}^{\b}(M)_n \leq \lambda_{\a'}^{\b}(M)_n$ for all $\a'\in\V(\a)$.
    \item Since $M$ is in dimension $< n$, there exists a finitely generated submodule $N$ of $M$ such that $\dim{\supp(M/N)} < n$.
    Then it is easy to see that $f_\a^\b(M)^n = f_\a^\b(N)^n$ and $f_\a^\b(M)_n = f_\a^\b(N)_n$.
    Since $N$ is finitely generated, there exists $t\in\mathbb{N}$ such that $\b^t \Gamma_{\b}(N) = 0$.
    Hence, from the exact sequence
    $$
    \cdots \to H_\a^i(\Gamma_{\b}(N)) \to H_\a^i(N) \to H_\a^i(N/{\Gamma_{\b}(N)}) \to \cdots,
    $$
    we see that $f_\a^\b(N)^n = f_\a^\b(N/{\Gamma_{\b}(N)})^n$ and $f_\a^\b(N)_n = f_\a^\b(N/{\Gamma_{\b}(N)})_n$.
\end{enumerate}
\end{rem}
\section{Relationship between $f_\a^\b(M)_n$ and $\lambda_\a^\b(M)_n$ over an almost Cohen-Macaulay ring}
In this section, we investigate the relationship between $f_\a^\b(M)_n$ and $\lambda_\a^\b(M)_n$ over an almost Cohen-Macaulay ring, and prove Theorem \ref{11} and Corollaries \ref{555}, \ref{12} and \ref{777}.
The discussion in this section is based on the one developed in \cite{BS}.
First, we introduce several notations.
\begin{dfn}
\rm{(cf. \cite[9.4.8]{BS})}
Let $M$ be a finitely generated $R$-module.
For each $t\in\mathbb{N}\cup\{-\infty\}$, let
$$
\U_t(M):= \{\p \in \spec{R} \mid \Rfd_{R_\p}M_\p \leq t\},
$$
let $\C_t(M) = \spec{R}\setminus\U_t(M)$, and let $\c_t(M) := \bigcap_{\p \in \C_t(M)}\p$.
\end{dfn}
We begin with the following proposition.
\begin{prop}\label{21}
\rm{(cf. \cite[Corollary 9.4.7]{BS})}
Let $R$ be an almost Cohen-Macaulay ring.
Then for each $t\in\mathbb{N}\cup\{-\infty\}$, the set $\U_t(M)$ is an open subset of $\spec{R}$ (in the Zariski topology) for all finitely generated $R$-modules $M$.
\end{prop}

\begin{proof}
Let $M$ be a finitely generated $R$-module.
If $t=-\infty$, then $\U_t(M)=\spec{R}\setminus\supp(M)$.
Hence $\U_t(M)$ is open.
Suppose that $t\geq0$.
Then the depth lemma \cite[Proposition 1.2.9]{BH} yields $\Rfd_{R_\p}(\syz_R^t M)_\p = \sup\{\Rfd_{R_\p}M_\p - t, 0\}$, and hence we have $\U_t(M)=\U_0(\syz_R^t M)$.
Thus, it is enough to show that $\U_0(M)$ is open.
Let $\p\in\U_0(M)$.
It follows from Proposition \ref{225} that there exists $s\in R\setminus\p$ such that $sH_I^i(M)=0$ for all ideals $I \subseteq R$ and all integers $i < \grade(I,R)$.
Again using Proposition \ref{225}, for each $\q\in\spec{R}$ with $s\notin\q$, we get $\Rfd_{R_\q}(M_\q) \leq 0$, and hence $\q\in\U_0(M)$.
Thus $\U_0(M)$ is open.
\end{proof}
Using the above proposition, we can prove the following corollary.
\begin{cor}\label{22}
\rm{(cf. \cite[Exrcise 9.4.9]{BS})}
Let $R$ be an almost Cohen-Macaulay ring and let $M$ be a finitely generated $R$-module, and let $S$ be a multiplicatively closed set of $R$.
Then $S^{-1}\c_t(M) = \c_t(S^{-1}M)$ for all $t\in\mathbb{N}\cup\{-\infty\}$.
\end{cor}
\begin{proof}
By Proposition \ref{21}, the set $\C_t(M)$ is a closed subset of $\spec{R}$.
Hence, we find $\p_1, \cdots ,\p_n \in \C_t(M)$ such that $\c_t(M) = \p_1\cap\cdots\cap\p_n$.
Then, since the localization commutes with the formation of finite intersections, we obtain $S^{-1}\c_t(M) = \c_t(S^{-1}M)$.
\end{proof}
The following two results play an important role in the proof of Theorem \ref{25}.
\begin{prop}\label{23}
\rm{(cf. \cite[Proposition 9.4.10]{BS})}
Let $t\in\mathbb{N}\cup\{-\infty\}$, and let $M$ be a finitely generated $R$-module.
Then there exists $k\in\mathbb{N}$ such that $\c_{t}(M)^k H_I^i(M) = 0$ for all ideals $I$ of $R$ and all integers $i<\grade(I,R) - t$.
\end{prop}
\begin{proof}
Let $\p\in\U_t(M)$.
Then we have $\Rfd_{R_\p}{M_\p} \leq t$.
Hence, it follows from Proposition \ref{226} that there exists $s_{\p}\in R\setminus\p$ such that $s_{\p}H_I^i(M)=0$ for all ideals $I \subseteq R$ and all integers $i < \grade(I,R) - t$.
Put $\g = \sum_{p\in\U_t(M)}s_{\p}R$.
Then $\g H_I^i(M)=0$ for all ideals $I \subseteq R$ and all integers $i < \grade(I,R) - t$.
Since $s_{\p}\in \g$ and $s_{\p}\notin \p$ for each $\p\in\U_t(M)$, we have $\V(\g) \subseteq \C_t(M)$.
Hence we get $\sqrt{\g} \subseteq \c_t(M)$.
Thus, there exists $k\in\mathbb{N}$ such that $\c_t(M)^k \subseteq \g$, and the assertion follows from this.
\end{proof}

\begin{lem}\label{24}
\rm{(cf. \cite[Lemma 9.4.14]{BS})}
Let $\q\in\V(\a)_{\geq n}$, and let $M$ be a finitely generated $R$-module.
Then we have $\b R_{\q} \subseteq \c_{\height\q - \lambda_\a^\b(M)_n}(M_\q)$.
\end{lem}
\begin{proof}
Put $t = \height\q - \lambda_\a^\b(M)_n$.
Suppose on the contrary, that there exists $\p R_\q \in \spec(R_\q)\setminus\U_t(M_\q)$ with $\b R_\q \nsubseteq \p R_\q$, and we look for a contradiction.
For $\s\in\U(\p)$, we have
$$
\lambda_\a^\b(M)_n\leq \lambda_{\a R_\q}^{\b R_\q}(M_\q) \leq  \depth(M_\q)_{\s R_\q} + \height{(\s R_\q + \q R_\q)/{\s R_\q}} \leq \height{\q} - (\depth{R_\s} - \depth{M_\s}).
$$
Hence we get $\depth{R_\s} - \depth{M_\s} \leq t$.
Thus, we obtain 
$$
\Rfd_{(R_\q)_{\p R_\q}}(M_\q)_{\p R_\q}=\Rfd_{R_\p}M_\p=\inf_{\s\in\U(\p)}\{\depth{R_\s} - \depth{M_\s} \}\leq t,
$$
and this inequality contradicts the fact that $\p R_\q \notin\U_t(M_\q)$.
\end{proof}
Now we can state and prove Theorem \ref{11}, which is the main result of this paper.
\begin{thm}\label{25}
Assume that $R$ is an almost Cohen-Macaulay ring.
Let $\a$ and $\b$ be ideals of $R$ such that $\b \subseteq \a$, and let $M$ be in dimension $< n$.
Then, 
$$
f_\a^\b(M)_n \geq \lambda_\a^\b(M)_n - \cmd{R}.
$$
\end{thm}
\begin{proof}
In view of Remark \ref{14} (2), we may assume that $M$ is finitely generated.
We divide the case by the value of $\dim{R/\a}$.\\
(The case $\dim{R/\a} \leq n$):
Set $T:=\{\p\in\V(\a) \mid \dim{R/\p}=n\}$.
Note that $T$ is finite.
Suppose that $T = \emptyset$.  Since $\dim{R/\a} \leq n$ and $\supp(H_{\a}^i(M)) \subseteq \V(\a)$, we have $\dim\supp(\b^t H_{\a}^i(M)) < n$ for all $t,i\in\mathbb{N}$.
Hence, we get $f_\a^\b(M)_n = \infty$, and there is nothing to show.
Let $T \neq \emptyset$ and write $T=\{\p_1,\dots,\p_m\}$. Put $t_j := \height{\p_j} - \lambda_\a^\b(M)_n$.
Note that $\p_1,\dots,\p_m$ are minimal primes of $\a$.
For all $1\leq j \leq m$, we have
$$
\grade(\a R_{\p_j},{R_{\p_j}}) - t_j =\lambda_\a^\b(M)_n -(\height{\p_j} - \depth{R_{\p_j}}) \geq\lambda_\a^\b(M)_n - \cmd{R},
$$
where the first equality follows from the equality $\sqrt{\a R_{\p_j}} = \p_j R_{\p_j}$.
Hence, by Proposition \ref{23} and the inequality above, there exists $s_j\in\mathbb{N}$ such that ${\c_{t_j}(M_{\p_j})}^{s_j} H_{\a R_{\p_j} }^i(M_{\p_j}) = 0$ for all $i < \lambda_\a^\b(M)_n -\cmd{R}$ and for all $1\leq j \leq m$.
Put $s:=\max\{s_1,\cdots,s_m\}$.
It follows from Lemma \ref{24} that
$$
(\b R_{\p_j})^s H_{\a R_{\p_j} }^i(M_{\p_j}) \subseteq {\c_{t_j}(M_{\p_j})}^{s} H_{\a R_{\p_j} }^i(M_{\p_j}) = 0
$$
for all $i < \lambda_\a^\b(M)_n -\cmd{R}$ and for all $1\leq j \leq m$.
This implies that $\dim\supp({\b}^s H_{\a }^i(M)) > n$ for all $i < \lambda_\a^\b(M)_n -\cmd{R}$, and thus we obtain $f_\a^\b(M)_n \geq \lambda_\a^\b(M)_n - \cmd{R}$.\\
(The case $\dim{R/\a} > n$):
Supposing contrarily that $f_\a^\b(M)_n < \lambda_\a^\b(M)_n - \cmd{R}$, we look for a contradiction.
Since $R$ is Noetherian, we can (and do) assume that $\a$ is a maximal element of the set $\{\c \mid \b \subseteq \c, f_{\c}^{\b}(M)_n < \lambda_{\c}^{\b}(M)_n - \cmd{R}\}$.
Let $\p_1,\cdots,\p_m$ be the distinct minimal primes of $\a$.
Then, there exists an integer $1\leq i \leq m$ such that $\dim{R/\a} = \dim{R/\p_i}$.
Since $\dim{R/\p_i} > n$, it follows from Corollary \ref{22} and Lemma \ref{24} that
$$
\b R_{\p_i} \subseteq \c_{\height{\p_i} - \lambda_{\a}^{\b} (M)_n}(M_{\p_i}) = (\c_{\height{\p_i} - \lambda_{\a}^{\b} (M)_n}(M))_{\p_i}.
$$
Then there exists $u\in R\setminus\p_i$ such that 
\begin{equation}\label{ABC}
    \b R_{u} \subseteq (\c_{\height{\p_i} - \lambda_\a^\b(M)_n}(M))_{u}.
\end{equation}
Since $\dim{R/\p_i} > n$, there exists $\q\in\V(\p_i)$ such that $\dim{R/\q} = n$. 
Let $v\in\q\setminus\p_i$ and $w\in\cap_{j\neq i}\p_j\setminus\p_i$.
Put $s:=uvw$.
By $\eqref{ABC}$ and Corollary \ref{22}, we have
\begin{equation}\label{DEF}
    \b R_{s} \subseteq (\c_{\height{\p_i} - \lambda_\a^\b(M)_n}(M))_{s} =\c_{\height{\p_i} - \lambda_\a^\b(M)_n}(M_s).
\end{equation}
Now, there exists $\p\in\V(\a)$ such that $s\notin\p$ and $\grade(\a R_s, R_s) = \depth(R_s)_{\p R_s} = \depth{R_\p}$.
Since $s\in\cap_{j\neq i}\p_j\setminus\p_i$, we have $\p_i\subseteq\p$ and $\height{\p}\geq\height{\p_i}$.
Hence we obtain
$$
\grade(\a R_s, R_s) - (\height{\p_i} - \lambda_\a^\b(M)_n)
\geq \depth{R_\p} - (\height{\p} - \lambda_\a^\b(M)_n)
\geq \lambda_\a^\b(M)_n.
$$
Thus, by Proposition \ref{23} and the inequality above, there exists $t\in\mathbb{N}$ such that 
$$
(\c_{\height{\p_i} - \lambda_\a^\b(M)_n}(M_s))^t H_{\a R_s}^j(M_s) = 0
\text{ for all } i < \lambda_\a^\b(M)_n -\cmd{R}.
$$
Hence, by $\eqref{DEF}$ and \cite[Theorem 4.2.1]{BS}, we obtain 
$$
\b^t H_{\a}^j(M_s) = 0
\text{ for all } i < \lambda_\a^\b(M)_n -\cmd{R}.
$$
Since $\a \subsetneq \a + Rs \subsetneq R$, we have $\lambda_{\a + Rs}^\b(M)_n -\cmd{R} \leq f_{\a + Rs}^\b(M)_n$ from the maximality of $\a$.
Since $\lambda_{\a}^\b (M)_n \leq \lambda_{\a + Rs}^\b (M)_n$, there exists $t'\in\mathbb{N}$ such that $\dim\supp({\b}^{t'} H_{\a + Rs}^i(M)) > n$ for all integers $i < \lambda_\a^\b(M)_n -\cmd{R}$.
Hence, from the exact sequence
$$
\cdots \rightarrow H_{\a}^i(M_s) \rightarrow  H_{\a}^i(M) \rightarrow H_{\a + Rs}^i(M) \rightarrow \cdots,
$$
we see that there exists $k\in\mathbb{N}$ such that $\dim\supp({\b}^k H_{\a }^i(M))> n$ for all integers $i < \lambda_\a^\b(M)_n -\cmd{R}$.
Therefore, we obtain $f_\a^\b(M)_n \geq \lambda_\a^\b(M)_n - \cmd{R}$, and this is a contradiction.
\end{proof}
The following is an immediate consequence of Theorem \ref{25}.
This is no other than Corollary \ref{555}.
\begin{cor}
Assume that $R$ is an almost Cohen-Macaulay ring.
Let $\a$ and $\b$ be ideals of $R$ such that $\b \subseteq \a$, and let $M$ be a finitely generated $R$-module.
Then, 
$$
f_\a^\b(M) \geq \lambda_\a^\b(M) - \cmd{R}.
$$
\end{cor}
\begin{proof}
Put $n = 0$ in Theorem \ref{25}.
\end{proof}
Theorem \ref{25} leads to the following corollary, which contains Corollaries \ref{12}, \ref{777}.
The assertion (2) refines \cite[Theorem 2.14]{Dou}.
\begin{cor}\label{26}
Let $\a$ and $\b$ be ideals of $R$ such that $\b \subseteq \a$, and let $M$ be in dimension $< n$.
\begin{enumerate}[\rm(1)]
    \item Assume that $R$ is a homomorphic image of an almost Cohen-Macaulay ring.
    Then,
    $$
    f_\a^\b(M)_n \geq \lambda_\a^\b(M)_n - 1.
    $$
    \item Assume that $R$ is a homomorphic image of a Cohen-Macaulay ring.
    Then,
     $$f_\a^\b(M)_n = \lambda_\a^\b(M)_n.
    $$
    \item Assume that $R$ is a homomorphic image of an almost Cohen-Macaulay ring.
    Then,
    $$
    f_\a^\b(M) \geq \lambda_\a^\b(M) - 1.
    $$
    \item Assume that $R$ is a homomorphic image of a Cohen-Macaulay ring.
    Then,
     $$f_\a^\b(M) = \lambda_\a^\b(M).
    $$
\end{enumerate}
\end{cor}
\begin{proof}
(1)
By \cite[Lemma 2.12, 2.13]{Dou}, we may assume that $R$ is itself an almost Cohen-Macaulay ring.
Since $\cmd{R} \leq 1$, we obtain $f_\a^\b(M)_n \geq \lambda_\a^\b(M)_n - 1$ by Theorem \ref{25}.\\
(2)
In the same way as in the proof above, we obtain $f_\a^\b(M)_n \geq \lambda_\a^\b(M)_n$.
Hence the assertion follows from Proposition \ref{100}.\\
(3) Put $n = 0$ in (1).\\
(4) Put $n = 0$ in (2).
\end{proof}

\begin{rem}
\begin{enumerate}[\rm(1)]
    \item In \cite{Ogo}, Ogoma constructed an example of a three-dimensional Noetherian normal domain $R$ which is not catenary.
    Since $R$ is normal, we have $\depth{R} \geq \min\{2,\dim{R}\} = 2$.
    We obtain $\cmd{R}=\dim{R}-\depth{R}\leq1$, and hence $R$ is an almost Cohen-Macaulay ring.
    However, since $R$ is not catenary, $R$ is not a homomorphic image of Cohen-Macaulay ring.
    Thus this is an example of a ring that is a homomorphic image of an almost Cohen-Macaulay ring, but not (a homorphic image of) a Cohen-Macaulay ring.
    \item By \cite[Example 3.1]{RN}, there exists a two-dimensional Noetherian local domain $(R, \m)$ such that $f_{\m}(R) = 1$ and $\lambda_{\m}(R) = 2$.
    Hence this satisfies the equation $f_{\m}(R) = \lambda_{\m}(R) - 1$.
\end{enumerate}
\end{rem}

Finally, we refine some results in \cite{DNEW} by using Theorem \ref{25}.
\begin{prop}\label{500}
\rm{(cf. \cite[Proposition 3.1]{DNEW})}
Assume that $R$ is a Cohen-Macaulay local ring.
Let $\a$ and $\b$ be ideals of $R$ such that $\b \subseteq \a$, and let $M$ be in dimension $<n$.
Then, 
$$
f_\a^\b(M)_n = \gamma_\a^\b(M)_n,
$$
where $\gamma_\a^\b(M)_n := \inf\{\depth M_\p + \height(\a + \p)/\p \mid \p \in \spec{R}\setminus\V(\b), \dim{R/(\a + \p) \geq n} \}$.
\end{prop}
\begin{proof}
In the proof of \cite[Proposition 3.1]{DNEW}, replace \cite[Theorem 2.10]{sar} with our Theorem \ref{25}.
Then the argument does work.
\end{proof}

\begin{cor}
\rm{(cf. \cite[Theorem 3.3]{DNEW})}
Assume that $R$ is a homomorphic image of a Cohen-Macaulay local ring.
Let $\a$ and $\b$ be ideals of $R$ such that $\b \subseteq \a$, and let $M$ be in dimension $<n$.
Then, 
$$
f_\a^\b(M)_n = \gamma_\a^\b(M)_n.
$$
\end{cor}
\begin{proof}
The assertion follows immediately from \cite[Lemma 2.12]{Dou}, \cite[Lemma 3.2]{DNEW} and Proposition \ref{500}.
\end{proof}
\section{Relationship between $f_\a^\b(M)_n$ and $f_\a^\b(M)^n$ over a ring of dimension at most 4}
In this section, we generalize results in \cite{BRS}, and prove Theorem \ref{13} and Corollary \ref{140}.
First, we begin with the following proposition, which is a generalization of \cite[Proposition 2.1]{BL}.
\begin{prop}\label{41}
Let $M$ be a finitely generated $R$-module, and let $i\in\mathbb{N}$.
Suppose that $H_{\a}^j(M)$ is in dimension $ < n$ for all $j < i$ and $N \subseteq H_{\a}^i(M)$ is in dimension $ < n$.
Then $\ass(H_{\a}^i(M)/N)_{\geq n}$ is finite.
\end{prop}
\begin{proof}
We use induction on $i$.
The case $i = 0$ is clear since $H_{\a}^0(M)$ is finitely generated.
Let $i > 0$.
It follows from \cite[Lemma 2.1.2 and Corollary 2.1.7 (iii)]{BS} that $H_{\a}^0(M/\Gamma_{\a}(M)) = 0$ and $H_{\a}^i(M/\Gamma_{\a}(M))\cong H_{\a}^i(M)$ for all $i > 0$.
We may replace $M$ by $M/\Gamma_{\a}(M)$ and assume $M$ is an $\a$-torsion-free $R$-module.
There exists an $M$-regular element $y\in\a$ by \cite[Lemma 2.1.1(ii)]{BS}.
Since $N$ is in dimension $ < n$, there exists a finitely generated submodule $N'$ of $N$ such that $\dim{\supp(N/N')} < n$.
Then it is easy to see that $\ass(H_{\a}^i(M)/N)_{\geq n} = \ass(H_{\a}^i(M)/N')_{\geq n}$.
It is enough to show that $\ass(H_{\a}^i(M)/N')_{\geq n}$ is finite.
Since $N' \subseteq H_{\a}^i(M)$ is a finitely generated $\a$-torsion $R$-module, there exists $n\in\mathbb{N}$ such that $y^n N' = 0$.
Put $x = y^n$.
The exact sequence $0 \to M \to M \to M/xM \to 0$ induces the exact sequence $H_{\a}^j(M) \to H_{\a}^j(M/xM) \to H_{\a}^{j+1}(M)$ for all $j\in\mathbb{N}$.
By \cite[proposition 2.12]{MNS}, we see that $H_{\a}^j(M/xM)$ is in dimension $< n$ for all $j < i-1$.
Moreover, by diagram chase, we get the following commutative diagram with exact rows and columns
in which $\delta$ is the connecting homomorphism and in which $\varepsilon$ is a natural map:
\[
\xymatrix@M=10pt{
H_{\a}^{i-1}(M) \ar[r]^-{\varepsilon} & H_{\a}^{i-1}(M/xM) \ar[d] \ar[r]^-{\delta} & H_{\a}^i(M) \ar[d] \ar[r]^-{x} & H_{\a}^i(M)\\
0 \ar[r]  & H_{\a}^{i-1}(M/xM)/{\delta^{-1}(N')} \ar[d] \ar[r]^-{\bar{\delta}} & H_{\a}^i(M)/N' \ar[d] \ar[r]^-{\bar{x}} &  H_{\a}^i(M)\\
& 0 & 0
}
\] 
Since $\ker(\delta) = \varepsilon(H_{\a}^{i-1}(M))$ and $N'$ are both in dimension $ < n$, so is $\delta^{-1}(N')$.
By induction, we obtain $\ass(T)_{\geq n}$ is finite, where $T:= H_{\a}^{i-1}(M/xM)/{\delta^{-1}(N')}$.
By the proof of \cite[Proposition 2.1]{BL}, we see that 
$$
\ass(H_{\a}^i(M)/N')_{\geq n} \subseteq \ass(T)_{\geq n} \cup \ass(N')_{\geq n},
$$
and hence $\ass(H_{\a}^i(M)/N')_{\geq n}$ is finite.
\end{proof}

\begin{lem}\label{42}
Let M be in dimension $ < n$, that is to say, let $M$ be an $R$-module such that there exists a finitely generated submodule $N$ of $M$ such that $\dim{\supp(M/N)} < n$.
If $H_{\a}^i(N)$ is in dimension $ < n$, then $H_{\a}^i(M)$ is in dimension $ < n$.
\end{lem}
\begin{proof}
Since $H_{\a}^i(N)$ is in dimension $ < n$, there exists a finitely generated submodule $L$ of $H_{\a}^i(N)$ such that $\dim{\supp(H_{\a}^i(N)/L)} < n$.
Now, we have the following commutative diagram in which $f$ is a natural map and in which $i$ and $j$ are inclusion maps:
\[
\xymatrix@M=8pt{
H_{\a}^i(N) \ar[r]^f   & H_{\a}^i(M) \\
L \ar[u]^i \ar[r] & f(L) \ar[u]^j
}
\] 
By the definitions of $N$ and $L$, we see that $f_{\p}$ and $i_{\p}$ are isomorphisms for all $\p\in\spec{R}_{\geq n}$.
It follows from the commutative diagram above that $j_{\p}$ is an isomorphism for all $\p\in\spec{R}_{\geq n}$.
Hence we have $\dim{\supp(H_{\a}^i(M)/f(L))} < n$, and thus $H_{\a}^i(M)$ is in dimension $ < n$.
\end{proof}
\begin{lem}\label{53}
Assume that $\ass(M)_{\geq n}$ is a nonempty finite set, that $\ass(M/K)_{\geq n}$ is a finite set for all finitely generated submodules $K$ of $M$, and that $M_{\p}$ is a finitely generated $R_{\p}$-module for all $\p\in\ass(M)_{\geq n}$.
Then there exists a finitely generated submodule $N$ of $M$ such that
$$
\bigcap_{\p\in\ass(M)_{\geq n}}\p \subsetneq \bigcap_{\p\in\ass(M/N)_{\geq n}}\p 
$$
\end{lem}
\begin{proof}
Set $\ass(M)_{\geq n} = \{\p_1,\cdots,\p_m\}$.
Then, for each $1\leq j \leq m$, there exists a finitely generated submodule $N_j$ of $M$ such that $M_{\p_j} = (N_j)_{\p_j}$.
Set $N = N_1 + \cdots + N_m$.
Then $N$ is a finitely generated submodule of $M$, and we have $(M/N)_{\p_j} = 0$ for all $1\leq j \leq m$.
We get $\ass(M)_{\geq n} \cap \ass(M/N)_{\geq n} = \emptyset$.
We will show that the assertion holds for this $N$.
If $\ass(M/N)_{\geq n} = \emptyset$, then it is clear.
Suppose that $\ass(M/N)_{\geq n} \neq \emptyset$.
First, we show that $\bigcap_{\p\in\ass(M)_{\geq n}}\p \subseteq \bigcap_{\p\in\ass(M/N)_{\geq n}}\p$.
Let $\s\in\ass(M/N)_{\geq n}$.
Then $\s\in\supp(M)_{\geq n}$, and hence there exists $\q\in\ass(M)_{\geq n}$ such that $\q\subseteq\s$.
Hence we have
$$
\bigcap_{\p\in\ass(M)_{\geq n}}\p \subseteq \q \subseteq \s.
$$
Next, we show that $\bigcap_{\p\in\ass(M)_{\geq n}}\p \neq \bigcap_{\p\in\ass(M/N)_{\geq n}}\p$.
Suppose on the contrary, that $\bigcap_{\p\in\ass(M)_{\geq n}}\p = \bigcap_{\p\in\ass(M/N)_{\geq n}}\p$, and we look for a contradiction.
Let $\p$ be a minimal element of $\ass(M)_{\geq n}$ with respect to the inclusion relation.
Since $\ass(M/N)_{\geq n}$ is finite by the assumption, there exists $\q\in\ass(M/N)_{\geq n}$ such that $\q \subseteq \p$.
Then, since $\q\in\supp(M)$, we get $\p = \q$.
This contradicts the fact that $\ass(M)_{\geq n} \cap \ass(M/N)_{\geq n} = \emptyset$.
\end{proof}
By using the two results above, we prove the following proposition, which gives us a short proof of \cite[Proposition 2.9]{MNS}.
\begin{prop}\label{43}
Let $M$ be in dimension $< n$, and let $i < f_{\a}(M)^n$.
Then $H_{\a}^i(M)$ is in dimension $< n$.
\end{prop}
\begin{proof}
In view of Remark \ref{14} (2) and Lemma \ref{42}, we may assume that $M$ is a finitely generated $R$-module.
We use induction on $i$.
The case $i = 0$ is clear since $M$ is finitely generated.
Let $i > 0$, and assume that $H_{\a}^j(M)$ is in dimension $< n$ for all $j < i$.
Suppose on the contrary, that $H_{\a}^i(M)$ is not in dimension $< n$, and we look for a contradiction.
Then $\ass(H_{\a}^i(M)/K)_{\geq n} \neq \emptyset$ for all finitely generated submodule $K$ of $H_{\a}^i(M)$.
By proposition \ref{41},  $\ass(H_{\a}^i(M)/K)_{\geq n}$ is finite for all finitely generated submodule $K$ of $H_{\a}^i(M)$.
Since $i < f_{\a}(M)^n$, it follows from \cite[Proposition 9.1.2]{BS} that $(H_{\a}^i(M))_{\p}$ is finitely generated for all $\p\in\spec(R)_{\geq n}$.
Hence, by repeatedly using Lemma \ref{53}, there exists the following sequence of ideals
$$
\bigcap_{\p\in\ass(H_{\a}^i(M))_{\geq n}}\p \subsetneq \bigcap_{\p\in\ass(H_{\a}^i(M)/N_0)_{\geq n}}\p \subsetneq \cdots \subsetneq \bigcap_{\p\in\ass(H_{\a}^i(M)/N_i)_{\geq n}}\p \subsetneq \cdots,
$$
where $N_i$ is a finitely generated submodule of $H_{\a}^i(M)$ for all $i \in \mathbb{N}$.
This contradicts the fact that $R$ is Noetherian.
\end{proof}

To state the results below, we introduce an invariant, which is a generalization of the grade of a module.
\begin{dfn}\label{45}
Let $M$ be in dimension $< n$.
We define the $n$-th grade of $\a$ on $M$ as follows.
$$
\grade(\a,M)_{\geq n} := \inf\{\grade((\a R_{\p},M_{\p})\mid\p\in\supp(M/{\a M})_{\geq n}\}.
$$
Note that $\grade(\a,M)_{\geq n} = \infty$ if and only if $\supp(M/{\a M})_{\geq n} = \emptyset$.
\end{dfn}
\begin{prop}\label{46}
Let $M$ be in dimension $< n$.
Then $\grade(\a,M)_{\geq n} = \inf\{i\mid\dim\supp(H_{\a}^i(M))\geq n\}$.
\end{prop}
\begin{proof}
Let $g:=\grade(\a,M)_{\geq n}$.
If $g = \infty$, then we have $\supp(M/{\a M})_{\geq n} = \emptyset$.
Since $\supp(H_{\a}^i(M)) \subseteq \supp(M/{\a M})$ for all $i\in\mathbb{N}$, we get $\dim{\supp(H_{\a}^i(M))} < n$ for all $i\in\mathbb{N}$.
Thus we have $g = \infty = \inf\{i\mid\dim{\supp(H_{\a}^i(M))\geq n}\}$.
Suppose that $g < \infty$.
If $i < g$, then we have $i < \grade(\a R_{\p},M_{\p})$ for all $\p\in\supp(M/{\a M})_{\geq n}$.
It follows from \cite[Theorem 6.2.7]{BS} that $(H_{\a}^i(M))_{\p} = 0$ for all $i < g$ and for all $\p\in\supp(M/{\a M})_{\geq n}$.
Thus we get $\dim{\supp(H_{\a}^i(M))} < n$ for all $i < g$.
Let $\p\in\supp(M/{\a M})_{\geq n}$ such that $g = \grade(\a R_{\p},M_{\p})$.
Then we have $(H_{\a}^g(M))_{\p} \neq 0$ by \cite[Theorem 6.2.7]{BS}, and hence we get $\dim\supp(H_{\a}^g(M)) \geq n$.
Thus the assertion follows.
\end{proof}
By Propositions \ref{41} and \ref{46}, we prove the following proposition, which is a generalization of \cite[Proposition 2.2]{BRS}.
\begin{prop}\label{47}
Let $M$ be in dimension $< n$, and $\supp(M/{\a M})_{\geq n}$ is nonempty.
Then the set $\ass(H_{\a}^g(M))_{\geq n}$ is finite.
\end{prop}
\begin{proof}
Let $g:=\grade(\a,M)_{\geq n} < \infty$.
Since $M$ is in dimension $< n$, there exists a finitely generated submodule $N$ of $M$ such that $\dim{\supp(M/N)} < n$.
Then, it is easy to see that $\ass(H_{\a}^g(M))_{\geq n} = \ass(H_{\a}^g(N))_{\geq n}$.
We may replace $M$ by $N$ and assume $M$ is finitely generated.
By Proposition \ref{46}, we have $\dim{\supp(H_{\a}^i(M))} < n$ for all $i < g$.
Hence $H_{\a}^i(M)$ is in dimension $< n$ for all $i < g$.
It follows from Proposition \ref{41} that $\ass(H_{\a}^g(M))_{\geq n}$ is finite.
\end{proof}
The following corollary is a generalization of \cite[Corollary 2.3]{BRS}.
\begin{cor}\label{48}
Let $\a$ and $\b$ be ideals of $R$ such that $\b \subseteq \a$, and let $M$ be in dimension $< n$. 
Assume that $\supp(M/{\a M})_{\geq n}$ is nonempty.
Set $g := \grade(\a,M)_{\geq n}$.
Then,
$$
f_{\a R_{\p}}^{\b R_{\p}} (M_\p) > g \text{ for all } \p\in\spec(R)_{\geq n} \quad \iff \quad f_{\a}^{\b}(M)_n > g.
$$
\end{cor}
\begin{proof}
In view of Proposition \ref{100}, it is enough to show that the first condition implies the second.
It follows from Proposition \ref{46} that $(H_{\a}^i(M))_{\p} = 0$ for all $\p\in\spec{R}_{\geq n}$ and for all $i < g$.
It suffices to prove that, if ${\b R_{\p}} \subseteq \sqrt{(0 :_{R_{\p}} H_{\a R_{\p}}^g(M_{\p}))}$ for all $\p\in\spec{R}_{\geq n}$, then $\dim{\supp({\b}^s H_{\a}^g(M))} < n$ for some $s\in\mathbb{N}$.
Now, $\ass(H_{\a}^g(M))_{\geq n}$ is finite by Proposition \ref{47}.
Set $\ass(H_{\a}^g(M))_{\geq n} = \{\p_1,\cdots\p_n\}$.
Then, by the assumption, there exists $s_{\p_i}$ such that $(\b R_{\p_i})^{s_{\p_i}}H_{\a R_{\p_i}}^g(M_{\p_i}) = 0$ for each $1 \leq i \leq n$.
Put $s = \max\{s_{\p_1},\cdots, s_{\p_n}\}$.
Then we have $({\b}^{s}H_{\a}^g(M))_{\p_i} = 0$ for all $1 \leq i \leq n$, and thus we obtain $\dim{\supp({\b}^s H_{\a}^g(M))} < n$.
\end{proof}
Corollary \ref{48} leads to the following corollary, which is a generalization of \cite[Corollary 2.4]{BRS}.
\begin{cor}\label{49}
Let $\a$ and $\b$ be ideals of $R$ such that $\b \subseteq \a$, and let $M$ be in dimension $ < n$.
Then $\ass(H_{\a}^1(M))_{\geq n}$ is finite, and 
$$
f_{\a R_{\p}}^{\b R_{\p}} (M_\p) > 1 \text{ for all } \p\in\spec(R)_{\geq n} \quad \iff \quad f_{\a}^{\b}(M)_n > 1.
$$
\end{cor}
\begin{proof}
Set $\overline{M} = M/\Gamma_{\a}(M)$.
Then, by \cite[Lemma 2.1.2 and Corollary 2.1.7 (iii)]{BS}, we have $H_{\a}^0(\overline{M}) = 0$ and $H_{\a}^1(M)\cong H_{\a}^1(\overline{M})$.
If $\dim{\supp(H_{\a}^1(M))} < n$, then we have $\ass(H_{\a}^1(M))_{\geq n} = \emptyset$.
Suppose that $\dim{\supp(H_{\a}^1(M))} \geq n$.
Then $\grade(\a, \overline{M})_{\geq n} = 1$ by Proposition \ref{46}.
Hence, by Proposition \ref{47}, we see that $\ass (H_{\a}^1(M))_{\geq n} = \ass (H_{\a}^1(\overline{M}))_{\geq n}$ is finite.
Next, we prove the latter assertion.
In view of Proposition \ref{100}, it is enough to show that the first condition implies the second.
Since $H_{\a}^0(M)$ is in dimension $ < n$, it suffices to prove that, if ${\b R_{\p}} \subseteq \sqrt{(0 :_{R_{\p}} H_{\a R_{\p}}^1(M_{\p}))}$ for all $\p\in\spec(R)_{\geq n}$, then $\dim{\supp({\b}^s H_{\a}^1(M))} < n$ for some $s\in\mathbb{N}$.
If $\dim{\supp(H_{\a}^1(M))} < n$, then it is clear.
If $\dim{\supp(H_{\a}^1(M))} \geq n$, then $\grade(\a, \overline{M})_{\geq n} = 1$ as we observed above.
Since $H_{\a}^1(M)\cong H_{\a}^1(\overline{M})$, we may replace $M$ by $\overline{M}$. 
Then the assertion follows from Corollary \ref{48}.
\end{proof}

The following theorem is a generalization of \cite[Theorem 2.6]{BRS}.
\begin{thm}\label{410}
Let $\a$ and $\b$ be ideals of $R$ such that $\b \subseteq \a$, and let $M$ be in dimension $ < n$.
Then,
$$
f_{\a R_{\p}}^{\b R_{\p}} (M_\p) > 2 \text{ for all } \p\in\spec(R)_{\geq n} \quad \iff \quad f_{\a}^{\b}(M)_n > 2.
$$
\end{thm}
\begin{proof}
In view of Proposition \ref{100}, it is enough to show that the first condition implies the second.
Suppose that $f_{\a R_{\p}}^{\b R_{\p}} (M_\p) > 2$ for all $ \p\in\spec{R}_{\geq n}$.
We must show that $f_{\a}^{\b}(M)_n > 2$.
In view of Remark \ref{14}, we may assume that $M$ is a finitely generated $\b$-torsion-free $R$-module.
Then there exists $y \in\b$ such that $y$ is $M$-regular by \cite[Lemma 2.1.1(ii)]{BS}.
By Corollary \ref{49}, we have $f_{\a}^{\b}(M)_n > 1$.
Hence it suffices to prove that $\dim{\supp(\b^s H_{\a}^2(M))} < n$ for some $s\in\mathbb{N}$.
By hypothesis, there exists $s_{\p}\in\mathbb{N}$ such that $(\b R_{\p})^{s_{\p}} H_{\a R_{\p}}^2(M_{\p}) = 0$ for each $\p\in\spec{R}_{\geq n}$.
The exact sequence $0\rightarrow M \xrightarrow{y^{s_{\p}}} M \to M/{y^{s_{\p}}M} \to 0$ yields the exact sequence
$$
H_{\a R_{\p}}^0(M_{\p}/{y^{s_{\p}}M_{\p}}) \to H_{\a R_{\p}}^1(M_{\p}) \xrightarrow{y^{s_{\p}}} H_{\a R_{\p}}^1(M_{\p}).
$$
From this exact sequence, we see that $H_{\a R_{\p}}^1(M_{\p})$ is finitely generated for all $\p\in\spec{R}_{\geq n}$, and hence $f_{\a}(M)^n > 1$.
Thus $H_{\a}^1(M)$ is in dimension $ < n$ by Proposition \ref{43}, and hence $\ass(H_{\a}^2(M))_{\geq n}$ is finite by Proposition \ref{41}.
Put $\ass(H_{\a}^2(M))_{\geq n} = \{\p_1,\cdots,\p_m\}$ and $s = \max\{s_{\p_1},\cdots,s_{\p_m}\}$.
Then we get $(\b^{s} H_{\a}^2(M))_{\p_i} = 0$ for all $1\leq i \leq m$, and this implies that $\dim{\supp(\b^s H_{\a}^2(M))} < n$.
\end{proof}
To state the next results, we introduce the following notation, which is a generalization of \cite[Definition 3.1]{BRS}.
\begin{dfn}\label{411}
Let $r\in\mathbb{N}$.
Let $\b$ be a ideal of $R$, and let $M$ be in dimension $ < n$.
We set
$$
\Pi_{n}^{\b}(M;r):= \{\a\mid \b \subseteq \a, f_{\a}^{\b}(M)_n \leq r \text{ and } f_{\a R_{\p}}^{\b R_{\p}} (M_\p) > r \text{ for all }  \p\in\spec{R}_{\geq n}\}.
$$
Note that if $\q\in\Pi_{n}^{\b}(M;r)$, then $\dim{R/\q} \geq n$.
\end{dfn}

The following lemma is a generalization of \cite[Lemma 3.2]{BRS}.
\begin{lem}\label{412}
Let $r\in\mathbb{N}$.
Let $\b$ be a ideal of $R$, and let $M$ be in dimension $ < n$.
Assume that the set $\Pi_{n}^{\b}(M;r)$ is not empty, and let $\q$ be a maximal member of $\Pi_{n}^{\b}(M;r)$ (with respect to inclusion). 
Then $\q$ is a non-maximal, prime ideal in $\supp(M)_{\geq n}$.
\end{lem}
\begin{proof}
Since $M$ is in dimension $ < n$, there exists a finitely generated submodule $N$ of $M$ such that $\dim{\supp(M/N)} < n$.
Then it is easy to see that $\Pi_{n}^{\b}(M;r) = \Pi_{n}^{\b}(N;r)$ and $\supp(M)_{\geq n} = \supp(N)_{\geq n}$.
Hence, we may replace $M$ by $N$ and assume $M$ is finitely generated.
Since $H_{\q}^i = H_{\sqrt{\q}}^i$ for all $i\in\mathbb{N}$, we have $\q = \sqrt{\q}$.
Also, since $H_{\q}^i(M) \cong H_{\q + (0:_R M)}^i(M)$ by \cite[Theorem 4.2.1]{BS}, we have $\q = \q + (0:_R M)$, so that $\q \supseteq (0:_R M)$.
First, we prove that $\q$ is a prime ideal.
Suppose on the contrary, that $\q$ is not a prime ideal, and we look for a contradiction.
Let $\p_1,\cdots,\p_t$ be the distinct minimal primes of $\q$.
Note that $t >1$ since $\q$ is not prime.
Let $s\in (\p_2\cap\cdots\cap\p_t)\setminus\p_1$.
Since $\q\in\Pi_{n}^{\b}(M;r)$, we have $f_{\q R_{\p}}^{\b R_{\p}} (M_\p) > r$ for all $ \p\in\spec{R}_{\geq n}$.
Hence, by Remark \ref{14} (1), we get $f_{(\q + Rs) R_{\p}}^{\b R_{\p}} (M_\p) > r$ and $f_{\p_1 R_{\p}}^{\b R_{\p}} (M_\p) > r$ for all $ \p\in\spec{R}_{\geq n}$.
As $\q \subsetneq \q + Rs$ and $\q \subsetneq \p_1$, we have $\q + Rs,\p_1 \notin \Pi_{n}^{\b}(M;r)$ from the maximality
of $\q$.
Hence we get $f_{\q + Rs}^{\b}(M)_n > r$ and $f_{\p_1}^{\b}(M)_n > r$.
Moreover, we have
$$
f_{\q}^{\b}(M_s)_n
=f_{\q R_s}^{\b R_s}(M_s)_n
=f_{\p_1 R_s}^{\b R_s}(M_s)_n
\geq f_{\p_1}^{\b}(M)_n
> r,
$$
where the first equality follows from \cite[4.2.1]{BS} and the second holds since $\sqrt{\q R_s} = \p_1 R_s$.
Hence the exact sequence $\cdots\to H_{\q + Rs}^i(M) \to H_{\q}^i(M) \to H_{\q}^i(M_s)\to \cdots$ yields $f_{\q}^{\b}(M)_n > r$, and this contradicts the fact that $\q\in\Pi_{n}^{\b}(M;r)$.
Thus $\q$ is prime, and so $\q \in \V((0:_R M))_{\geq n} = \supp(M)_{\geq n}$.
Next, we show that $\q$ is not a maximal ideal.
Suppose on the contrary, that $\q$ is a maximal ideal, and we look for a contradiction.
Since $\dim{R/\q} \geq n$, we only need to show that in the case $n = 0$.
Fix $t\in\mathbb{N}$.
Suppose that $(\b R_{\q})^{t}H_{\q R_{\q}}^i(M_{\q}) = 0$.
Then, since $\supp({\b}^t H_{\q}^i(M)) \subseteq \V(\q) = \{\q\}$, we have $\supp({\b}^t H_{\q}^i(M)) = \phi$.
Hence we get ${\b}^t H_{\q}^i(M) = 0$.
Thus we see that $f_{\q}^{\b}(M) = f_{\q R_{\q}}^{\b R_{\q}}(M_{\q})$, and this contradicts the fact that $\q\in\Pi_0^{\b}(M;r)$.
\end{proof}
By using Lemma \ref{412}, we can prove the following theorem, which is a generalization of \cite[Theorem 3.10]{BRS}.
\begin{thm}\label{413}
Assume that $\dim{R}$ is finite.
Let $r\in\mathbb{N}$ be such that $r \geq \dim{R} - 1$.
Let $\a$ and $\b$ be ideals of $R$ such that $\b \subseteq \a$, and let $M$ be in dimension $< n$.
Then,
$$
f_{\a R_{\p}}^{\b R_{\p}} (M_\p) > r \text{ for all } \p\in\spec(R)_{\geq n} \quad \iff \quad f_{\a}^{\b}(M)_n > r.
$$
\end{thm}
\begin{proof}
In view of Proposition \ref{100}, it is enough to show that the first condition implies the second.
We suppose that the set $\Pi_{n}^{\b}(M;r)$ is not empty, and look for a contradiction.
In view of Remark \ref{14} (2), we may assume that $M$ is a finitely generated $\b$-torsion-free $R$-module.
Then there exists $x\in\b$ such that $y$ is $M$-regular by \cite[Lemma 2.1.1(ii)]{BS}.
Let $\q$ be a maximal member of $\Pi_{n}^{\b}(M;r)$.
By Lemma \ref{412}, $\q$ is non-maximal, prime ideal in $\supp(M)_{\geq n}$.
Since $\q\in\Pi_{n}^{\b}(M;r)$, we have $f_{\q R_{\p}}^{\b R_{\p}} (M_\p) > r$ for all $ \p\in\spec{R}_{\geq n}$.
Hence we get
\begin{align*}
\lambda_{\q}^{\b}(M)_n
&= \inf\{\lambda_{\q R_\p}^{\b R_\p}(M_\p) \mid \p\in\spec{R}_{\geq n} \}\\
&\geq \inf\{f_{\q R_\p}^{\b R_\p}(M_\p) \mid \p\in\spec{R}_{\geq n}\}\\
&> r.
\end{align*}
Since $\q\in\supp(M)_{\geq n}$, there exists $\p\in\ass(M)_{\geq n}$ such that $\p \subseteq \q$, and since $x\in\b$ is a non-zero divisor on $M$, we must have $\p\notin\V(\b)$.
Therefore, we obtain
$$
\height{\q} \geq \height{\q/\p} = \depth (M_\q)_{\p R_\q} + \height{\q R_\q/\p R_\q} = \lambda_{\q R_\p}^{\b R_\q}(M_\q) \geq \lambda_{\q}^{\b}(M)_n \geq r + 1 \geq \dim{R}.
$$
Since $\q$ is not a maximal ideal of $R$, this is a contradiction.
\end{proof}

Now, we can state and prove Theorem \ref{13}, which is the main result in this section.
This result is a generalization of \cite[Corollary 3.12]{BRS}.
\begin{cor}\label{414}
Let $R$ be a ring such that $\dim{R} \leq 4$.
Let $\a$ and $\b$ be ideals of $R$ such that $\b \subseteq \a$, and let $M$ be in dimension $< n$.
Then,
$$
f_\a^\b(M)_n = f_\a^\b(M)^n
$$
\end{cor}
\begin{proof}
This is an immediate consequence of Corollary \ref{49} and Theorems \ref{410}, \ref{413}.
\end{proof}
To show our next result, we state a remark and a lemma.
\begin{rem} \label{999}
Let $\a$ and $\b$ be ideals of $R$ such that $\b \subseteq \a$, and let $M$ be a finitely genrated $R$-module.
If $\spec{R} \setminus \V(\b) = \emptyset$, then there exists $n\in\mathbb{N}$ such that $\b^n = 0$.
Hence we get $f_{\a}^{\b}(M) = \infty$.
\end{rem}

\begin{lem} \label{415}
Let $(R,\m)$ be a local ring.
Let $\a$ and $\b$ be ideals of $R$ such that $\b \subseteq \a$, and let $M$ be a finitely generated $R$-module.
Then $\lambda_{\a}^{\b}(M) = \infty$ implies $f_{\a}^{\b}(M) = \infty$.
\end{lem}
\begin{proof}
In view of Remark \ref{999}, we may assume that $\spec{R} \setminus \V(\b) \neq \emptyset$.
If $\a = R$, then $f_{R}^{\b}(M) = \infty$.
Let $\a\neq R$.
Suppose that $\lambda_{\a}^{\b}(M) = \infty$. 
Then we have $ \depth M_{\p} + \height(\a + \p)/{\p} = \infty$ for all $\p\in\spec{R}\setminus\V(\b)$.
As $\a + \p \subseteq \m$, we get $\height(\a + \p)/\p < \infty$.
Hence, we have $\depth M_{\p} = \infty$, and $M_{\p} = 0$.
Thus, we see that $\supp(M) \subseteq \V(\b)$ and hence there exists $n\in\mathbb{N}$ such that $\b^n \subseteq (0:_R M)$.
We obtain $\b^n H_{\a}^i(M) = 0$ for all $i\in\mathbb{N}$, and hence $f_{\a}^{\b}(M) = \infty$.
\end{proof}
Finally, we prove Corollary \ref{140}, which states that if $R$ is a ring such that $\dim{R} \leq 2$, then the assertion of Corollary \ref{26} (2) always holds.
\begin{cor}
Let $R$ be a ring such that $\dim{R} \leq 2$.
Let $\a$ and $\b$ be ideals of $R$ such that $\b \subseteq \a$, and let $M$ be in dimension $< n$.
Then,
$$f_\a^\b(M)_n \geq \lambda_\a^\b(M)_n - 1.
$$
\end{cor}
\begin{proof}
In view of Remark \ref{14}, we may assume that $M$ is finitely generated.
If $\lambda_\a^\b(M)_n < \infty$, then it is easy to see that $\lambda_\a^\b(M)_n \leq \dim{R} \leq 2$.
Hence, we get $f_\a^\b(M)_n \geq \lambda_\a^\b(M)_n - 1$.
Suppose that $\lambda_\a^\b(M)_n = \infty$.
Then $\lambda_{\a R_\p}^{\b R_\p}(M_\p) = \infty$ for all $\p\in\spec{R}_{\geq n}$.
By Lemmma \ref{415}, we have $f_{\a R_\p}^{\b R_\p}(M_\p) = \infty$ for all $\p\in\spec{R}_{\geq n}$.
It follows from Corollary \ref{414} that $f_\a^\b(M)_n = f_\a^\b(M)^n = \infty$.
\end{proof}
\begin{ac}
The author appreciates the support of his advisor Ryo Takahashi.
The author also thanks Takeshi Kawasaki for helpful comments.
\end{ac}

\end{document}